\documentclass[11pt,letterpaper,reqno]{amsart}
\usepackage{amssymb}
\usepackage{amsmath}
\usepackage{amsthm}
\usepackage{amsfonts}
\usepackage{enumitem}
\usepackage{booktabs}
\usepackage{graphicx}
\usepackage{xcolor}
\usepackage[T1]{fontenc}
\usepackage{doi}
\usepackage{comment}
\usepackage{pgfplots}
\pgfplotsset{compat=1.18}
\usepackage{hyperref}
\hypersetup{pdfstartview={FitH}}
\usepackage{bookmark}
\usepackage[capitalize,noabbrev]{cleveref}

\newtheorem{theorem}{Theorem}[section]
\newtheorem{lemma}[theorem]{Lemma}
\newtheorem{proposition}[theorem]{Proposition}
\newtheorem{corollary}[theorem]{Corollary}

\newtheorem{conjecture}[theorem]{Conjecture}
\theoremstyle{definition}

\crefname{theoremletter}{Theorem}{Theorems}
\Crefname{theoremletter}{Theorem}{Theorems}

\DeclareMathOperator{\tr}{tr}
\DeclareMathOperator{\spanop}{span}
\newcommand{\R}{\mathbb R}
\newcommand{\Sph}{\mathbb S}

\begin{document}

\title[A proof of the Ashbaugh--Benguria conjecture]{A proof of the Ashbaugh--Benguria conjecture for reciprocal sums of Neumann eigenvalues}

\author[Y.~He]{Yixin He}
\address{School of Mathematical Sciences, Fudan University, Shanghai 200433, P. R. China}
\email{yixin.he717@gmail.com}

\author[Y.~Li]{Yanyang Li}
\address{School of Mathematics, Southeast University, Nanjing 211189, P. R. China}
\email{liyanyang1219@gmail.com}

\author[Q.~Tang]{Quanyu Tang${}^\ast$}
\thanks{${}^\ast$Corresponding author.}
\address{School of Mathematical Sciences, University of Science and Technology of China, Hefei 230026, P. R. China}
\email{tangquanyu827@gmail.com}

% 35P15  	Estimates of eigenvalues in context of PDEs
% 49R05  	Variational methods for eigenvalues of operators 
% 49Q10  	Optimization of shapes other than minimal surfaces


\begin{abstract}
We prove the Ashbaugh--Benguria conjecture for bounded domains with Lipschitz boundary in $\mathbb R^m$. More precisely, among all bounded Lipschitz domains of fixed volume, the ball minimizes the sum of the reciprocals of the first $m$ nonzero Neumann eigenvalues. Equality is attained precisely by balls.
\end{abstract}

\maketitle

\section{Introduction}

Throughout the paper, a domain means a connected open set. Unless stated
otherwise, \(\Omega\subset\mathbb R^m\) is a bounded domain with Lipschitz
boundary. We write
\[
        0=\mu_0(\Omega)<\mu_1(\Omega)\leq\mu_2(\Omega)\leq\cdots
\]
for the Neumann eigenvalues of \(\Omega\), counted with multiplicity. These are the eigenvalues of
\[
\begin{cases}
        -\Delta u=\mu u & \text{in }\Omega,\\
        \dfrac{\partial u}{\partial\nu}=0 & \text{on }\partial\Omega,
\end{cases}
\]
where \(\nu\) denotes the outward unit normal, defined almost everywhere on \(\partial\Omega\), and the boundary condition is understood weakly. Equivalently, the eigenvalues are characterized by the Neumann quadratic form on \(H^1(\Omega)\).

The classical point of departure is the Szeg\H{o}--Weinberger
isoperimetric inequality. If \(\Omega\subset\R^m\) is a smooth bounded
domain and \(B\subset\R^m\) is the ball with \(|B|=|\Omega|\), then
\[
        \mu_1(\Omega)\le \mu_1(B).
\]
Equality holds only when \(\Omega\) is a ball \cite{Sze54,Wei56}. In dimension two, Szeg\H{o}'s approach also gives the sharp two-term reciprocal inequality
\[
        \frac1{\mu_1(\Omega)}+\frac1{\mu_2(\Omega)}
        \ge
        \frac2{\mu_1(D)}
\]
for simply connected planar domains, where \(D\) is the disk with
\(|D|=|\Omega|\).

Since the first nonzero Neumann eigenvalue of a Euclidean ball has
multiplicity \(m\), its first nontrivial eigenspace is \(m\)-dimensional.
It is therefore natural to seek sharp inequalities that reflect this
whole eigenspace, rather than only the single eigenvalue \(\mu_1\). One
such formulation is the harmonic-mean, or reciprocal-sum, problem for the
first \(m\) nonzero Neumann eigenvalues.

Ashbaugh and Benguria formulated the following long-standing conjecture~\cite{AB93}; see also Ashbaugh's list of open problems~\cite{Ashbaugh99} and Henrot's monograph~\cite{Henrot06}.

\begin{conjecture}[\cite{AB93}]\label{conj:AB-euclidean}
Let \(m\geq2\), and let \(\Omega\subset\R^m\) be a bounded domain with smooth boundary. If \(B\subset\R^m\) is the ball satisfying \(|B|=|\Omega|\), then
\begin{equation}\label{eq:intro-AB-euclidean}
        \sum_{k=1}^m \frac1{\mu_k(\Omega)}
        \ge
        \frac{m}{\mu_1(B)}.
\end{equation}
Equality holds if and only if \(\Omega\) is a ball.
\end{conjecture}

The conjecture implies the Szeg\H{o}--Weinberger inequality. Indeed, since \(\mu_1(\Omega)\leq \mu_k(\Omega)\) for \(k=1,\ldots,m\), we have \(\sum_{k=1}^m 1/\mu_k(\Omega) \leq m/\mu_1(\Omega)\). Thus \eqref{eq:intro-AB-euclidean} gives \(m/\mu_1(\Omega)\ge m/\mu_1(B)\), and hence \(\mu_1(\Omega)\le \mu_1(B)\).

Variational principles and estimates for sums of reciprocal eigenvalues
go back at least to Hersch \cite{Her61} and were further developed, for
example, by Hile and Xu \cite{HX93}. Related reciprocal-sum inequalities
for Laplacian and Steklov-type problems were studied by Dittmar
\cite{Dit02}, Brock \cite{Bro01}, and Enache--Philippin \cite{EP13}.

Ashbaugh and Benguria~\cite{AB93} also proved the universal \(m\)-term lower bound
\[
        \sum_{k=1}^m \frac1{\mu_k(\Omega)}
        \ge
        \frac{m}{m+2}\left(\frac{|\Omega|}{\omega_m}\right)^{2/m},
\]
where \(\omega_m\) denotes the volume of the unit ball in \(\mathbb R^m\).
This estimate is weaker than the conjectured sharp bound
\eqref{eq:intro-AB-euclidean}. For the sharp harmonic-mean problem, the main general progress toward
the conjecture had reached the first \(m-1\) nonzero Neumann eigenvalues:
Xia and Wang proved the sharp \((m-1)\)-term inequality for bounded domains
with smooth boundary in \(\mathbb R^m\) \cite[Theorem~1.1]{XW23}; thus
their result supports, but does not prove, the full \(m\)-term
Ashbaugh--Benguria conjecture. Recent related developments include the
\((m-1)\)-term Gaussian analogue of Gao and Wang for origin-symmetric
Lipschitz domains in Gauss space \cite{GaoWang26}, Witten-Laplacian
analogues of the \((m-1)\)-term reciprocal inequality due to Chen and Mao
\cite{CM24}, and lower bounds for reciprocal sums of Neumann eigenvalues
due to Eddaoudi~\cite{Edd25}. The remaining difficulty in the Euclidean
conjecture is precisely the last reciprocal term.

We prove Conjecture~\ref{conj:AB-euclidean}. The main result is the
following theorem.

\begin{theorem}\label{thm:main-euclidean}
Let \(m\geq2\), and let \(\Omega\subset\R^m\) be a bounded domain with Lipschitz boundary. Let \(B\subset\R^m\) be the ball satisfying \(|B|=|\Omega|\). Then
\[
        \sum_{k=1}^m\frac1{\mu_k(\Omega)}
        \geq
        \frac{m}{\mu_1(B)}.
\]
Equality holds if and only if \(\Omega\) is a ball.
\end{theorem}

\subsection{Proof strategy}
We use the same trial functions, center of mass, and mass transplantation
as in Weinberger's original proof of the inequality for $\mu_1$
\cite{Wei56}. Let \(B_R\) be the ball with \(|B_R|=|\Omega|\), let
\(\lambda=\mu_1(B_R)\), and write the first nonzero Neumann eigenspace
of \(B_R\) in the form \(g(r)\theta_i\), \(i=1,\ldots,m\). We extend the
radial factor constantly outside \(B_R\), obtaining \(G\). A Weinberger
center \(p\) is chosen so that the transplanted functions
\[
        P_i(x)=G(|x-p|)\frac{x_i-p_i}{|x-p|}
\]
have zero mean on \(\Omega\). Define the mass and stiffness matrices by
\[
        M_{ij}=\int_\Omega P_iP_j\,dx,
        \qquad
        K_{ij}=\int_\Omega \nabla P_i\cdot\nabla P_j\,dx.
\]
The trace form of the Ritz principle gives
\[
        \sum_{k=1}^m\frac1{\mu_k(\Omega)}
        \geq\tr(K^{-1}M).
\]

The key point is to keep the full \(m\)-dimensional trial space coupled,
rather than estimating the transplanted coordinates separately. In the
argument of Xia and Wang \cite{XW23}, scalar variational estimates for
suitably chosen coordinate trial functions yield the sharp
$(m-1)$-term inequality. Here we retain the full mass and stiffness
matrices. Write $B_R(p)=p+B_R$ for the ball of radius $R$ centered at $p$. Applying Weinberger's mass transplantation globally, by
comparing the integrals over $B_R(p)\setminus\Omega$ and
$\Omega\setminus B_R(p)$, gives
\[
        M\succeq aI+cZ,
        \qquad
        K\preceq \lambda aI-dZ,
\]
where $a,c,d>0$ and
\[
        Z=\int_\Omega\theta\theta^\top\,dx
          -\int_{B_R(p)}\theta\theta^\top\,dx,
        \qquad \theta=\frac{x-p}{|x-p|}.
\]
Since $\tr(\theta\theta^\top)=1$ almost everywhere, the equal-volume
hypothesis gives $\tr Z=0$. Corollary~\ref{cor:matrix-Jensen} then yields
$\tr(K^{-1}M)\geq m/\lambda$. Equality in the matrix inequality and the
mass comparison force $\Omega$ to agree with $B_R(p)$ up to a null
set; the Lipschitz boundary assumption gives equality of the domains.

\subsection{Organization of the paper}
Section~\ref{sec:matrix} develops the matrix tools used in the proof.
Section~\ref{sec:radial} recalls the standard ball trial functions.
Section~\ref{sec:euclidean-proof} proves
Theorem~\ref{thm:main-euclidean} using the global matrix comparison.
Section~\ref{sec:future} discusses the corresponding problem for sums of
Neumann eigenvalues.

\section{Matrix tools for reciprocal eigenvalue sums}\label{sec:matrix}

For real symmetric matrices, $A\succeq B$ means that $A-B$ is positive
semidefinite, and $A\succ0$ means that $A$ is positive definite. We use
two elementary facts. If $A,B\succeq0$, then
\[
        \tr(AB)=\tr(B^{1/2}AB^{1/2})\geq0,
\]
since $B^{1/2}AB^{1/2}$ is positive semidefinite. Also, a positive
semidefinite matrix with trace zero is the zero matrix, because all
its eigenvalues are nonnegative.

We first recall a trace form of Hersch's variational principle
\cite{Her61,HX93}.

\begin{lemma}\label{lem:Ritz}
Let $\Omega\subset\R^m$ be a bounded Lipschitz domain. Let
$P_1,\ldots,P_m\in H^1(\Omega)$ be linearly independent in
$L^2(\Omega)$ and satisfy
\[
        \int_\Omega P_i\,dx=0,
        \qquad i=1,\ldots,m.
\] 
Define
\[
        M_{ij}=\int_\Omega P_iP_j\,dx,
        \qquad K_{ij}=\int_\Omega\nabla P_i\cdot\nabla P_j\,dx.
\]
If $K\succ0$, then
\[
        \sum_{k=1}^m\frac1{\mu_k(\Omega)}\geq\tr(K^{-1}M).
\]
\end{lemma}

\begin{proof}
The Gram matrix $M$ is positive definite. Let
$V=\spanop\{P_1,\ldots,P_m\}$, and let
$0<\eta_1\leq\cdots\leq\eta_m$ be the min--max values obtained by
restricting the Neumann Rayleigh quotient to $V$. Since
$V\subset H^1(\Omega)\cap1^\perp$, where $1^\perp$ denotes the
zero-mean subspace of $L^2(\Omega)$, the min--max principle gives
$\mu_k(\Omega)\leq\eta_k$ for $k=1,\ldots,m$; see
\cite[Theorem~3.1.11]{LMP23}. The $\eta_k$ are the generalized
eigenvalues of $K\beta=\eta M\beta$. Equivalently,
\[
        K\beta=\eta M\beta
        \quad\Longleftrightarrow\quad
        K^{-1}M\beta=\frac1\eta\beta.
\]
Thus $1/\eta_k$ are the eigenvalues of $K^{-1}M$. Consequently,
\[
        \sum_{k=1}^m\frac1{\mu_k(\Omega)}
        \geq\sum_{k=1}^m\frac1{\eta_k}
        =\tr(K^{-1}M).\qedhere
\]
\end{proof}

The following proposition is the finite-dimensional ingredient that
closes the full $m$-term estimate. We first give its normalized form,
with only one scalar parameter.

\begin{proposition}\label{prop:matrix-Jensen}
Let $c>-1$, and let $Z$ be a real symmetric $m\times m$ matrix with
$\tr Z=0$. Suppose that $M$ and $K$ are real symmetric matrices with
\[
        M\succeq0,\qquad M\succeq I+cZ,
        \qquad 0\prec K\preceq I-Z.
\]
Then
\begin{equation}\label{eq:matrix-normalized-goal}
        \tr(K^{-1}M)\geq m.
\end{equation}
Equality holds if and only if $Z=0$ and $M=K=I$.
\end{proposition}

\begin{proof}
We first replace $K$ by its upper bound and then $M$ by its lower
bound, in order to estimate $\tr(K^{-1}M)$ from below.
Set $T=I-Z$. Since $0\prec K\preceq T$, we have $T\succ0$ and
$K^{-1}\succeq T^{-1}$; inversion reverses the Loewner order on
positive definite matrices \cite[Proposition~V.1.6]{Bhatia97}.
The trace-product fact recalled above, applied to
$K^{-1}-T^{-1}$ and $M$, and then to $T^{-1}$ and $M-I-cZ$, gives
\begin{equation}\label{eq:matrix-chain}
        \tr(K^{-1}M)
        \geq\tr(T^{-1}M)
        \geq\tr\bigl(T^{-1}(I+cZ)\bigr).
\end{equation}
This order of comparison uses $M\succeq0$ but does not require
$I+cZ$ to be positive semidefinite.
Since $T=I-Z$, the symmetric matrices $T^{-1}$ and $Z$ commute and
hence can be simultaneously diagonalized. If $z_1,\ldots,z_m$ are the
eigenvalues of $Z$, then $\sum_i z_i=0$ and $z_i<1$, because
$T\succ0$. Hence
\[
        \tr\bigl(T^{-1}(I+cZ)\bigr)
        =\sum_{i=1}^m\frac{1+cz_i}{1-z_i}
        =m+(1+c)\sum_{i=1}^m f(z_i),
        \qquad
        f(z)=\frac{z}{1-z}.
\]
The function $f$ is strictly convex on $(-\infty,1)$, since
$f''(z)=2/(1-z)^3>0$. Jensen's inequality gives
\[
        \frac1m\sum_{i=1}^m f(z_i)
        \geq f\left(\frac1m\sum_{i=1}^m z_i\right)=f(0)=0.
\]
Since $1+c>0$, this proves \eqref{eq:matrix-normalized-goal}.

If equality holds in \eqref{eq:matrix-normalized-goal}, then equality
holds throughout \eqref{eq:matrix-chain} and in Jensen's inequality.
Since $f$ is strictly convex, the equality case of Jensen's inequality gives \(z_1=\cdots=z_m\).
Since their sum is zero, $Z=0$ and $T=I$.
Equality in the second inequality of \eqref{eq:matrix-chain} now gives
$\tr(M-I)=0$. The matrix $M-I$ is positive semidefinite, so the
trace-zero fact stated above implies $M=I$. Equality in the first
inequality gives $\tr(K^{-1}-I)=0$, whence $K^{-1}=I$ by the same
fact. Thus $K=I$. The converse is immediate.
\end{proof}

\begin{corollary}\label{cor:matrix-Jensen}
Let $a,d,\lambda>0$ and $c>-d/\lambda$. Let $Z$ be a real symmetric
$m\times m$ matrix with $\tr Z=0$. Suppose that $M$ and $K$ are
real symmetric matrices satisfying
\[
        M\succeq0,\qquad M\succeq aI+cZ,
        \qquad 0\prec K\preceq\lambda aI-dZ.
\]
Then
\begin{equation}\label{eq:matrix-goal}
        \tr(K^{-1}M)\geq\frac{m}{\lambda}.
\end{equation}
Equality holds if and only if $Z=0$, $M=aI$, and $K=\lambda aI$.
\end{corollary}
\begin{proof}
Apply Proposition~\ref{prop:matrix-Jensen} to
\[
        \widehat M=\frac{M}{a},\qquad
        \widehat K=\frac{K}{\lambda a},\qquad
        \widehat Z=\frac{d}{\lambda a}Z,\qquad
        \widehat c=\frac{\lambda c}{d}>-1.
\]
Since $\tr(\widehat K^{-1}\widehat M)=\lambda\tr(K^{-1}M)$,
Proposition~\ref{prop:matrix-Jensen} yields \eqref{eq:matrix-goal},
together with the stated equality characterization.
\end{proof}

\section{The standard radial trial functions}\label{sec:radial}

Separation of variables for the Neumann spectrum of a Euclidean ball
is standard; see \cite[Section~2]{XW23} and
\cite[{\S}1.2.3 and {\S}5.3.1]{LMP23}. For
$B_R=\{x\in\R^m:|x|<R\}$ and $\lambda=\mu_1(B_R)$, the first
nonzero eigenspace has dimension $m$ and is spanned by
\[
        u_i(r,\theta)=g(r)\theta_i,\qquad i=1,\ldots,m,
\]
where $r=|x|$, $\theta=x/|x|\in\Sph^{m-1}$ for $x\neq0$, and $g$ is
the radial factor corresponding to the first nonzero Neumann
eigenvalue.
Choosing the sign of $g$ so that $g'(0)>0$, the radial factor satisfies
\[
        g''+\frac{m-1}{r}g'
        +\left(\lambda-\frac{m-1}{r^2}\right)g=0,
        \qquad 0<r<R,
\]
with $g(0)=0$ and $g'(R)=0$. More explicitly, put $\gamma=m/2$, let $J_\gamma$ be the Bessel
function of the first kind of order $\gamma$, and let $p_{\gamma,1}$
be the first positive zero of $(s^{1-\gamma}J_\gamma(s))'$. Then $\lambda=\left(p_{\gamma,1}/R\right)^2$, and, up to multiplication by a positive constant,
\[
        g(r)=r^{1-\gamma}
        J_\gamma\left(\frac{p_{\gamma,1}}{R}r\right).
\]
The regular expansion at zero is
\begin{equation}\label{eq:radial-expansion}
        g(r)=\alpha r+O(r^3),\qquad
        g'(r)=\alpha+O(r^2),\qquad
        \alpha=g'(0)>0,
        \qquad r\downarrow0.
\end{equation}
Extend $g$ by
\begin{equation}\label{eq:constant-extension}
        G(r)=
        \begin{cases}
          g(r),&0\leq r\leq R,\\
          g(R),&r\geq R.
        \end{cases}
\end{equation}
We recall the standard properties of Weinberger's radial trial
function; see \cite[p.~635]{Wei56} and
\cite[Lemma~2.2]{XW23}.

\begin{lemma}\label{lem:euclidean-radial}
The radial factor satisfies
\[
        g(r)>0\quad(0<r\leq R),
        \qquad
        g'(r)>0\quad(0<r<R).
\]
The extension $G$ is $C^1$ and nondecreasing on $[0,\infty)$, while
$G(r)/r$ is nonincreasing on $(0,\infty)$. Consequently, $G^2$ is
nondecreasing on $[0,\infty)$, $G^2/r^2$ is nonincreasing on
$(0,\infty)$, and $G'(r)=0$ for $r\geq R$.
\end{lemma}

Figure~\ref{fig:radial-G} illustrates the radial factor and its constant
extension in dimension two.

\begin{figure}[htbp]
\centering
\begin{tikzpicture}
\begin{axis}[
  width=0.62\textwidth,height=4.8cm,
  axis lines=left,xmin=0,xmax=1.55,ymin=0,ymax=1.15,
  xlabel={$r$},ylabel={$G(r)$},
  xtick={0,0.5,1,1.5},ytick={0,0.5,1},
  xticklabels={$0$,$0.5$,$R=1$,$1.5$},
  tick label style={font=\small},label style={font=\small},
  clip=false]
\addplot[black,thick] coordinates {(0.00,0.00000000) (0.01,0.01582072) (0.02,0.03163742) (0.03,0.04744608) (0.04,0.06324268) (0.05,0.07902320) (0.06,0.09478363) (0.07,0.11051996) (0.08,0.12622820) (0.09,0.14190437) (0.10,0.15754446) (0.11,0.17314452) (0.12,0.18870059) (0.13,0.20420871) (0.14,0.21966496) (0.15,0.23506541) (0.16,0.25040616) (0.17,0.26568333) (0.18,0.28089303) (0.19,0.29603144) (0.20,0.31109471) (0.21,0.32607904) (0.22,0.34098064) (0.23,0.35579574) (0.24,0.37052062) (0.25,0.38515155) (0.26,0.39968486) (0.27,0.41411688) (0.28,0.42844397) (0.29,0.44266255) (0.30,0.45676904) (0.31,0.47075991) (0.32,0.48463164) (0.33,0.49838076) (0.34,0.51200384) (0.35,0.52549747) (0.36,0.53885829) (0.37,0.55208298) (0.38,0.56516823) (0.39,0.57811080) (0.40,0.59090748) (0.41,0.60355509) (0.42,0.61605052) (0.43,0.62839068) (0.44,0.64057251) (0.45,0.65259304) (0.46,0.66444931) (0.47,0.67613841) (0.48,0.68765748) (0.49,0.69900371) (0.50,0.71017435) (0.51,0.72116667) (0.52,0.73197803) (0.53,0.74260579) (0.54,0.75304742) (0.55,0.76330039) (0.56,0.77336225) (0.57,0.78323060) (0.58,0.79290309) (0.59,0.80237744) (0.60,0.81165140) (0.61,0.82072279) (0.62,0.82958950) (0.63,0.83824944) (0.64,0.84670062) (0.65,0.85494108) (0.66,0.86296893) (0.67,0.87078234) (0.68,0.87837953) (0.69,0.88575880) (0.70,0.89291849) (0.71,0.89985702) (0.72,0.90657285) (0.73,0.91306451) (0.74,0.91933061) (0.75,0.92536981) (0.76,0.93118082) (0.77,0.93676242) (0.78,0.94211348) (0.79,0.94723290) (0.80,0.95211967) (0.81,0.95677281) (0.82,0.96119145) (0.83,0.96537475) (0.84,0.96932194) (0.85,0.97303235) (0.86,0.97650532) (0.87,0.97974030) (0.88,0.98273678) (0.89,0.98549433) (0.90,0.98801259) (0.91,0.99029124) (0.92,0.99233006) (0.93,0.99412888) (0.94,0.99568758) (0.95,0.99700614) (0.96,0.99808457) (0.97,0.99892299) (0.98,0.99952153) (0.99,0.99988044) (1.00,1.00000000)};
\addplot[black,thick] coordinates {(1,1) (1.5,1)};
\addplot[gray,dashed] coordinates {(1,0) (1,1)};
\end{axis}
\end{tikzpicture}
\caption{The radial trial function in dimension $m=2$, normalized by
$R=1$ and $G(1)=1$, together with its constant extension for $r\geq1$.}
\label{fig:radial-G}
\end{figure}

We will also use the following elementary identities for the ball.
By rotational symmetry and the Neumann eigenvalue equation,
\begin{equation}\label{eq:ball-matrices}
        \left(\int_{B_R}u_i u_j\,dx\right)_{i,j}=aI,
        \qquad
        \left(\int_{B_R}\nabla u_i\cdot\nabla u_j\,dx\right)_{i,j}
        =\lambda aI,
        \qquad
        a=\frac1m\int_{B_R}G(|x|)^2\,dx>0.
\end{equation}
Indeed, rotational symmetry gives
$\int_{B_R}u_i u_j\,dx=a\delta_{ij}$, while
$\sum_{i=1}^m u_i^2=g^2=G^2$ on $B_R$ determines the value of $a$.
Testing $-\Delta u_i=\lambda u_i$ against $u_j$ gives the second
identity.

\section{Proof of the main result}\label{sec:euclidean-proof}

We use Weinberger's standard orthogonality lemma
\cite[p.~635]{Wei56}; see also \cite[Corollary~2]{Lau20}.

\begin{lemma}\label{lem:weinberger-center}
Let $\Omega\subset\R^m$ be a bounded domain, and let
$G:[0,\infty)\to[0,\infty)$ be continuous and bounded, with $G(0)=0$,
and eventually equal to a positive constant. Then there exists
$p\in\R^m$ such that
\begin{equation}\label{eq:weinberger-center}
        \int_\Omega G(|x-p|)\frac{x-p}{|x-p|}\,dx=0,
\end{equation}
where the integrand is interpreted as zero at $x=p$.
\end{lemma}

\begin{proof}[Proof of Theorem~\ref{thm:main-euclidean}]
Choose $R>0$ so that $|B_R|=|\Omega|$, put $\lambda=\mu_1(B_R)$,
and let $G$ be the extension in \eqref{eq:constant-extension}.

\medskip
\noindent\textbf{Step 1. Trial space and the Ritz bound.}
Choose $p$ by Lemma~\ref{lem:weinberger-center}, and set
$B=B_R(p):=p+B_R$. The center $p$ need not belong to $\Omega$.
For $x\ne p$, write
\[
        r=|x-p|,\qquad
        \theta=\frac{x-p}{r},\qquad
        P_i(x)=G(r)\theta_i,\quad i=1,\ldots,m,
\]
and set $P_i(p)=0$. The formula
\[
        P_i(p+v)=\frac{G(|v|)}{|v|}v_i
\]
together with \eqref{eq:radial-expansion} gives
$P_i(p+v)=\alpha v_i+O(|v|^3)$. For $v\ne0$, differentiation gives
\[
        \partial_j P_i(p+v)
        =\frac{G(r)}{r}\delta_{ij}
          +\left(G'(r)-\frac{G(r)}{r}\right)\theta_i\theta_j
        =\alpha\delta_{ij}+O(r^2),
        \qquad r=|v|\downarrow0.
\]
Thus $P_i$ extends as a $C^1$ function at $p$. Since $g'(R)=0$, the constant extension $G$ is $C^1$ at $r=R$,
and hence $P_i$ is $C^1$ across the sphere $|x-p|=R$.
The functions and their first derivatives are bounded on $\Omega$,
so $P_i\in H^1(\Omega)$. Moreover,
\eqref{eq:weinberger-center} gives $\int_\Omega P_i\,dx=0$ for $i=1,\ldots,m$.

Let $M$ and $K$ be the mass and stiffness matrices of these functions,
as in Lemma~\ref{lem:Ritz}. These functions are linearly independent. Indeed, let
$\xi=(\xi_1,\ldots,\xi_m)\in\R^m$ and set $P_\xi=\sum_{i=1}^m\xi_iP_i$. Suppose that $P_\xi=0$ almost everywhere on $\Omega$.
By continuity, $P_\xi$ vanishes everywhere on $\Omega$.
Since $G(r)>0$ for $r>0$,
for $\xi\ne0$ this would place the nonempty open set $\Omega$ in the
hyperplane $\xi\cdot(x-p)=0$, which is impossible. Hence the mass
matrix $M$ is positive definite. The stiffness matrix $K$ is also
positive definite: if $\xi^\top K\xi=0$, then $P_\xi$ has zero weak
gradient and is constant on the connected domain $\Omega$.
Its zero mean makes that constant zero, so linear independence gives
$\xi=0$. Lemma~\ref{lem:Ritz} therefore yields
\begin{equation}\label{eq:euclidean-Ritz}
        \sum_{k=1}^m\frac1{\mu_k(\Omega)}\geq\tr(K^{-1}M).
\end{equation}

\medskip
\noindent\textbf{Step 2. Global mass transplantation.}
For $x\ne p$, set
\[
        S(x)=\theta\theta^\top,\qquad
        q(r)=G'(r)^2,\qquad h(r)=\frac{G(r)^2}{r^2}.
\]
The values of $S$ at $p$ are immaterial for integration; $h$ has the
continuous extension $h(0)=\alpha^2$. Differentiating the trial
functions gives
\begin{equation}\label{eq:global-MK}
        M=\int_\Omega G(r)^2S(x)\,dx,
        \qquad
        K=\int_\Omega\bigl[q(r)S(x)+h(r)(I-S(x))\bigr]\,dx.
\end{equation}
In \eqref{eq:global-MK}, \(S(x)\) and \(I-S(x)\) are complementary
orthogonal projections and, in particular, are positive semidefinite.
By translation invariance, \eqref{eq:ball-matrices} gives the
corresponding mass and stiffness matrices $aI$ and $\lambda aI$
on the comparison ball $B$. Define
\begin{equation}\label{eq:global-Z}
        Z=\int_\Omega S(x)\,dx-\int_B S(x)\,dx
          =\int_{\Omega\setminus B}S(x)\,dx
             -\int_{B\setminus\Omega}S(x)\,dx.
\end{equation}
Since $\tr S(x)=1$ almost everywhere,
\begin{equation}\label{eq:global-Z-trace}
        \tr Z=|\Omega|-|B|=0.
\end{equation}
This is where the equal-volume hypothesis enters the matrix argument.

Put $E_-=B\setminus\Omega$ and $E_+=\Omega\setminus B$.
The equal-volume hypothesis gives $|E_+|=|E_-|$. Set
\[
        c=G(R)^2>0,\qquad d=h(R)=\frac{c}{R^2}>0.
\]
By the definition of $G$ and Lemma~\ref{lem:euclidean-radial},
$G^2=c$ outside $B$ and $G^2\leq c$ inside $B$, so
\begin{equation}\label{eq:global-mass-defect}
\begin{aligned}
        M-aI-cZ
        &=\int_{E_+}(G(r)^2-c)S(x)\,dx
           +\int_{E_-}(c-G(r)^2)S(x)\,dx\\
        &=\int_{E_-}(c-G(r)^2)S(x)\,dx\succeq0.
\end{aligned}
\end{equation}
Thus
\begin{equation}\label{eq:euclidean-M-lower}
        M\succeq aI+cZ.
\end{equation}

For the stiffness matrix, Lemma~\ref{lem:euclidean-radial} gives
$q=0$ outside $B$, while the monotonicity of $h$ gives $h\geq d$ inside
$B$ and $h\leq d$ outside $B$.
Using $|E_+|=|E_-|$ and \eqref{eq:global-Z}, we have
\[
        \int_{E_+}(I-S)\,dx-\int_{E_-}(I-S)\,dx
        =(|E_+|-|E_-|)I-Z=-Z.
\]
Consequently, using the corresponding stiffness identity on $B$ and
the preceding relation, we obtain
\begin{equation}\label{eq:global-stiffness-defect}
\begin{aligned}
        \lambda aI-dZ-K
        &=\int_{E_-}\bigl[q(r)S(x)+(h(r)-d)(I-S(x))\bigr]\,dx\\
        &\quad+\int_{E_+}(d-h(r))(I-S(x))\,dx
        \succeq0.
\end{aligned}
\end{equation}
Every integrand on the right-hand side of
\eqref{eq:global-stiffness-defect} is positive semidefinite. Hence
\begin{equation}\label{eq:euclidean-K-upper}
        K\preceq\lambda aI-dZ.
\end{equation}
These are the desired global matrix versions of Weinberger's
mass-transplantation estimates.

\medskip
\noindent\textbf{Step 3. The sharp bound.}
By \eqref{eq:global-Z-trace}, $\tr Z=0$. Together with
$M\succeq0$, $K\succ0$, and $a,c,d,\lambda>0$,
Corollary~\ref{cor:matrix-Jensen}, applied to
\eqref{eq:euclidean-M-lower} and \eqref{eq:euclidean-K-upper}, gives
\[
        \tr(K^{-1}M)\geq\frac{m}{\lambda}.
\]
Together with \eqref{eq:euclidean-Ritz}, this proves the inequality.

\medskip
\noindent\textbf{Step 4. Equality.}
If $\Omega$ is a ball, then its first nonzero Neumann eigenvalue has
multiplicity $m$, so equality holds. Conversely, suppose equality holds in the theorem. Since
\[
        \sum_{k=1}^m\frac1{\mu_k(\Omega)}
        \geq\tr(K^{-1}M)
        \geq\frac{m}{\lambda},
\]
equality must hold in both inequalities.
Corollary~\ref{cor:matrix-Jensen} gives $Z=0$, $M=aI$ and $K=\lambda aI$. Taking the trace in \eqref{eq:global-mass-defect} and using
$\tr S(x)=1$ almost everywhere gives
\[
        0=\int_{B\setminus\Omega}(c-G(r)^2)\,dx.
\]
Since $g$ is strictly increasing on $(0,R)$ and $g(0)=0$, the integrand
is strictly positive everywhere in $B$. Thus $|B\setminus\Omega|=0$;
equal volume also gives $|\Omega\setminus B|=0$.

We finally use the Lipschitz boundary assumption to upgrade the
measure-theoretic equality to equality of the domains. Since
$\Omega$ is open and $|\Omega\setminus B|=0$, we have
$\Omega\subset B$: otherwise a sufficiently small neighborhood of a
point of $\Omega\setminus B$ would have positive measure outside $B$.

Conversely, suppose that $x\in B\setminus\Omega$. Since
$|B\setminus\Omega|=0$ and $x$ is an interior point of $B$, we must
have $x\in\partial\Omega$; otherwise a sufficiently small ball
centered at $x$ would be contained in $B\setminus\Omega$. Since
$\Omega$ is a Lipschitz domain, its complement has positive measure
in every sufficiently small neighborhood of a boundary point.
Taking such a neighborhood contained in $B$ contradicts
$|B\setminus\Omega|=0$. Hence $B\subset\Omega$, and therefore
$\Omega=B_R(p)$.
\end{proof}

\section{Open problems and further directions}\label{sec:future}

The ball does not in general maximize the sum of the first $m$ nonzero
Neumann eigenvalues at fixed volume. As observed by Ashbaugh and Benguria
\cite[p.~569]{AB93}, a $2\times1$ rectangle already gives a
counterexample in dimension two.

It is nevertheless instructive to see where the present matrix method
breaks down for this functional. For the same trial space, the min--max
principle gives
\[
        \sum_{k=1}^m\mu_k(\Omega)
        \leq \tr(M^{-1}K).
\]
Write $L=aI+cZ$ and $T=\lambda aI-dZ$, where
$a,c,d,\lambda>0$ are the geometric parameters from the proof of
Theorem~\ref{thm:main-euclidean}. Our estimates give
$M\succeq L$ and $K\preceq T$, but do not guarantee that $L$ is
positive definite, so the inverse Loewner-order comparison cannot be
applied in general.

Even if $L\succ0$, one only obtains
\[
        \tr(M^{-1}K)\leq\tr(L^{-1}K)\leq\tr(L^{-1}T).
\]
If $z_i$ are the eigenvalues of $Z$, then
\[
\begin{aligned}
        \tr(L^{-1}T)
        &=m\lambda-(\lambda c+d)\sum_{i=1}^m
          \frac{z_i}{a+cz_i}\\
        &\geq m\lambda.
\end{aligned}
\]
Indeed, $L\succ0$ implies $a+cz_i>0$ for every $i$, and the last
inequality follows from the concavity of
$z/(a+cz)$ on $a+cz>0$ and from $\sum_i z_i=0$. Thus this comparison does
not give the upper bound $m\lambda$ needed for a ball-maximization
argument. It is therefore natural to ask for sharp upper bounds for this sum and
for a characterization of the maximizing domains.

\section*{Acknowledgements}

The authors are grateful to the referees for their careful reading and
constructive suggestions, which substantially improved the presentation
of the paper, in particular the global formulation of the
mass-transplantation argument and the discussion of further directions.

The authors acknowledge the use of AI tools for exploratory discussions and language refinement during the preparation of this work. More specifically, suggestions generated by ChatGPT contributed to the formulation and organization of the argument that appears in Section~\ref{sec:matrix} of this paper. All mathematical statements, proofs, and references were independently checked, revised, and finalized by the authors, who take full responsibility for the content of the paper.

\section*{Declarations}

\subsection*{Funding.}
The authors did not receive support from any organization for the submitted work.

\subsection*{Competing interests.}
The authors declare that they have no relevant financial or non-financial interests to disclose.

\subsection*{Data availability.}
Data sharing is not applicable to this article, as no datasets were generated or analyzed during the current study.

\end{document}